\documentclass{mcom-l}
\usepackage{amscd,amssymb,accents,graphicx,xcolor}
\usepackage{atbegshi}
\AtBeginShipoutFirst{\vspace{-.25in}\small\color{red!50!black}\centerline{To appear in:
Mathematics of Computation, 2013.}}
\def\S{\mathcal S}
\def\T{\mathcal T}
\def\P{\mathcal P}
\def\Q{\mathcal Q}
\def\H{\mathcal H}
\def\J{\mathcal J}
\def\R{\mathbb R}
\def\N{\mathbb N}
\def\0{\ring}
\def\grad{\operatorname{grad}}
\def\curl{\operatorname{curl}}
\renewcommand{\div}{\operatorname{div}} 

\newcommand\tr{\operatorname{tr}}
\newcommand\x{\times}
\newcommand{\ldeg}{\operatorname{ldeg}}
\newcommand{\supp}{\operatorname{supp}} 

\newtheorem{thm}{Theorem}[section] 
\newtheorem{lemma}[thm]{Lemma} 
\newtheorem{prop}[thm]{Proposition} 
 
\begin{document}
\title{Finite element differential forms on cubical meshes}

\author{Douglas N. Arnold}
\address{Department of Mathematics, University of Minnesota, Minneapolis,
Minnesota 55455}
\email{arnold@umn.edu}
\thanks{The work of the first author was supported in part by NSF grant
DMS-1115291.}

\author{Gerard Awanou}
\address{Department of Mathematics, Statistics, and Computer Science, M/C 249.
University of Illinois at Chicago, 
Chicago, IL 60607-7045}
\email{awanou@uic.edu}  
\thanks{The work of the second author was supported in part by
NSF grant DMS-0811052 and the Sloan Foundation.}


\begin{abstract}
We develop a family of finite element spaces of differential forms defined
on cubical meshes in any number of dimensions.
The family contains elements of all polynomial degrees and all
form degrees.  In two
dimensions, these include the serendipity finite elements and the
rectangular BDM elements.  In three
dimensions they include a recent generalization of the serendipity spaces,
and new $H(\curl)$ and $H(\div)$ finite element spaces.
Spaces in the family can be combined to give finite element subcomplexes of the de~Rham
complex which satisfy the basic hypotheses of the finite element exterior calculus,
and hence can be used for stable discretization of a variety of problems.
The construction and properties of the spaces are established in a uniform manner
using finite element exterior calculus.
\end{abstract}

\maketitle

\section{Introduction}
In this paper we develop a family of finite element spaces $\S_r\Lambda^k(\T_h)$ of differential forms,
where $\T_h$ is a mesh
of cubes in $n\ge 1$ dimensions, $r\ge1$ is the polynomial degree, and $0\le k\le n$ is
the form degree.  Thus, in $3$ dimensions,  the space $\S_r\Lambda^k(\T_n)$ is a finite
element subspace of the Hilbert space
$H^1$, $H(\curl)$, $H(\div)$, or $L^2$, according to whether $k=0$, $1$, $2$, or $3$.
For $n=1$ or $2$, the spaces were previously known, while in three (or more) dimensions,
they are mostly new.  Specifically, our construction yields a new family
of $H(\curl)$ elements and a new family of $H(\div)$ elements on cubical meshes in three dimensions.
Our treatment in an exterior calculus 
framework allows all the
spaces and their properties to be developed together.  The spaces combine together in
complexes satisfying the basic hypotheses of the finite element exterior calculus
\cite{bulletin}.  This means that, in addition to their use individually, they can
be used in pairs, $\S_{r+1}\Lambda^{k-1}(\T_h)\x\S_r\Lambda^k(\T_h)$ in a variety
of mixed finite element applications, with stability and convergence following
from the abstract theory of \cite{bulletin}.  Element diagrams
for some of the spaces are shown in Figure~{\ref{elts}}.  The dimension
of the shape function spaces are given in Theorem~\ref{unisolv} below
and are tabulated for $n\le 4$ and $r\le 7$ in Table~\ref{dims}.

\begin{figure}
\leftline{\hspace{1.35in}$r=1$\hspace{1.13in}$r=2$\hspace{1.13in}$r=3$}%
\vspace{.15in}%
\centerline{%
 \raise.5in\hbox{$k=0$}\qquad
 \raise.5in\hbox{$\S_1\Lambda^0$}\includegraphics[width=1in]{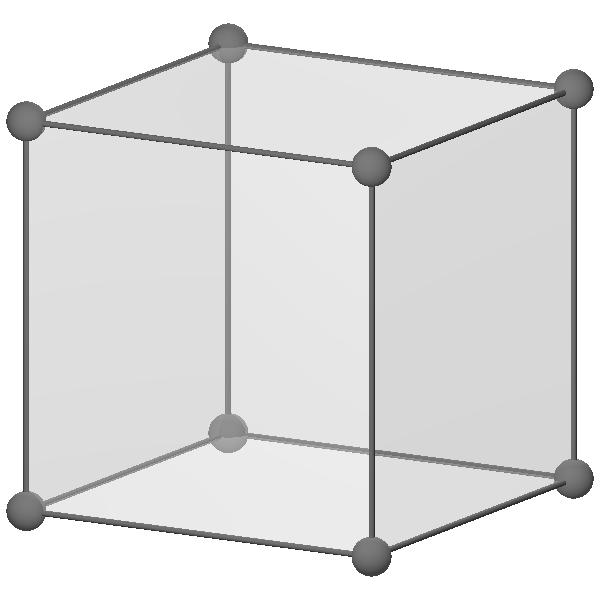}\quad%
 \raise.5in\hbox{$\S_2\Lambda^0$}\includegraphics[width=1in]{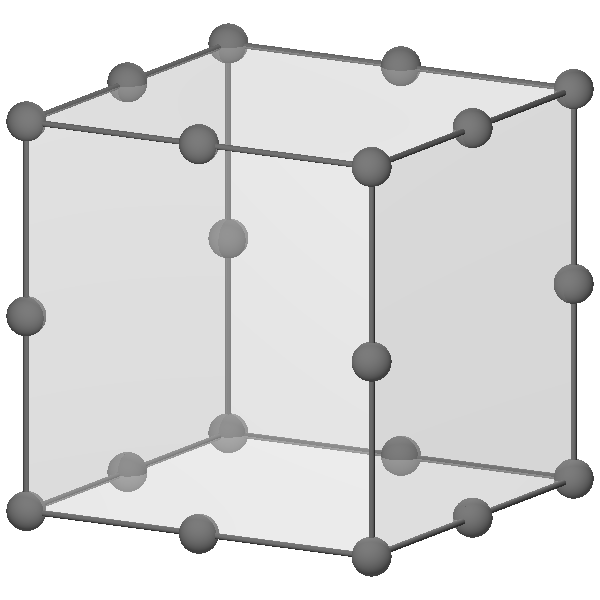}\quad%
 \raise.5in\hbox{$\S_3\Lambda^0$}\includegraphics[width=1in]{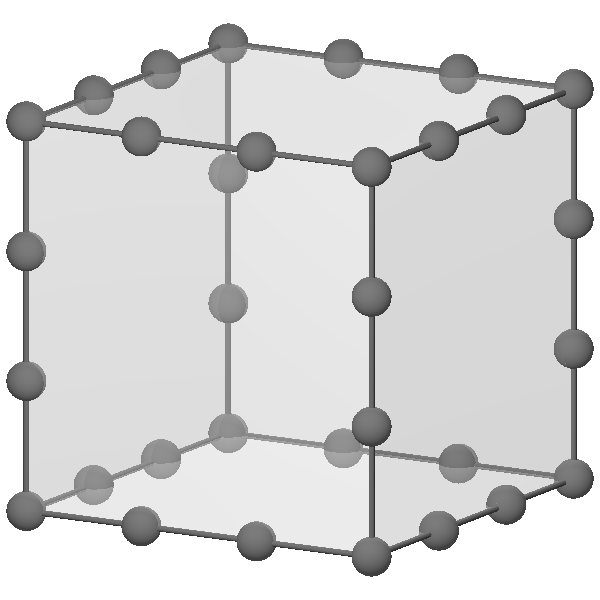}%
}
\centerline{%
 \raise.5in\hbox{$k=1$}\qquad
 \raise.5in\hbox{$\S_1\Lambda^1$}\includegraphics[width=1in]{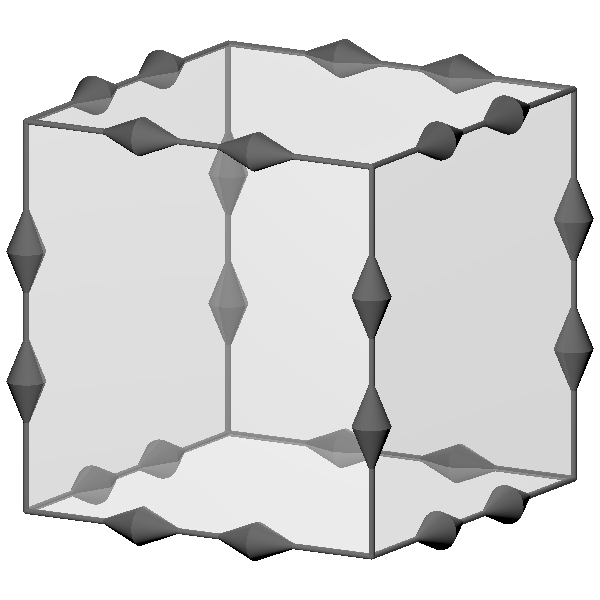}\quad%
 \raise.5in\hbox{$\S_2\Lambda^1$}\includegraphics[width=1in]{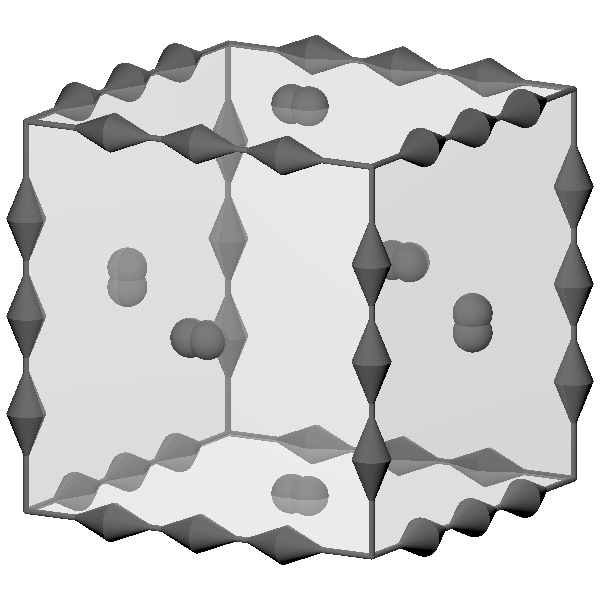}\quad%
 \raise.5in\hbox{$\S_3\Lambda^1$}\includegraphics[width=1in]{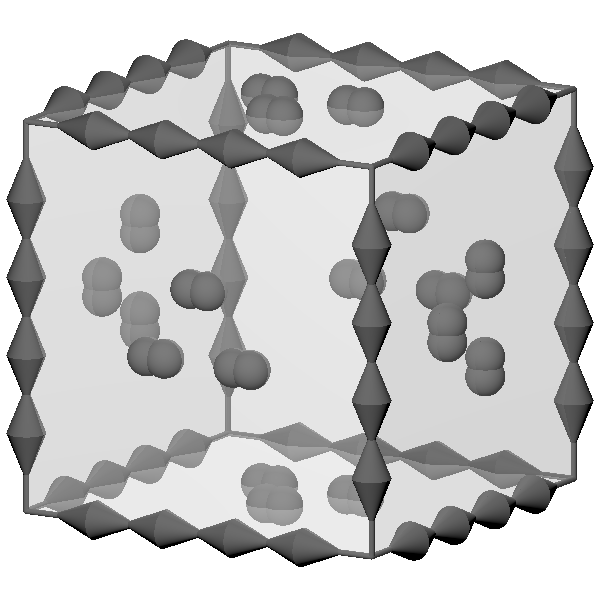}%
}
\centerline{%
 \raise.5in\hbox{$k=2$}\qquad
 \raise.5in\hbox{$\S_1\Lambda^2$}\includegraphics[width=1in]{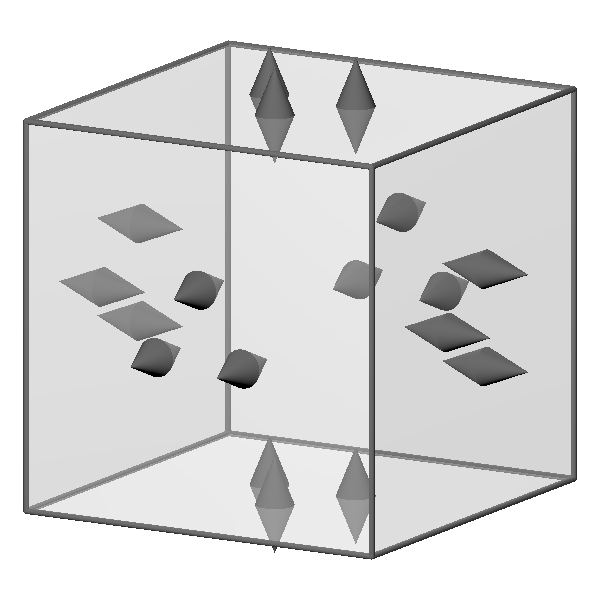}\quad%
 \raise.5in\hbox{$\S_2\Lambda^2$}\includegraphics[width=1in]{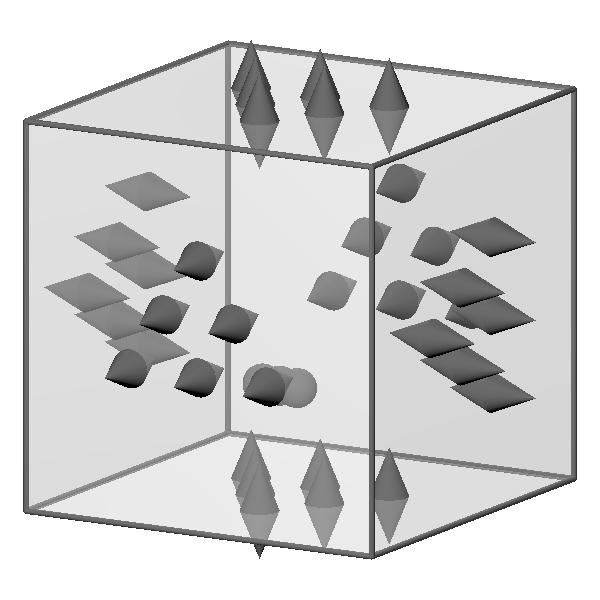}\quad%
 \raise.5in\hbox{$\S_3\Lambda^2$}\includegraphics[width=1in]{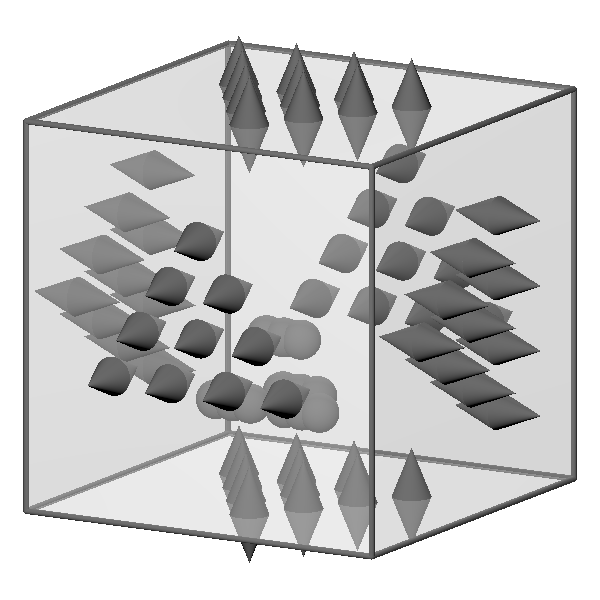}%
}
\centerline{%
 \raise.5in\hbox{$k=3$}\qquad
 \raise.5in\hbox{$\S_1\Lambda^3$}\includegraphics[width=1in]{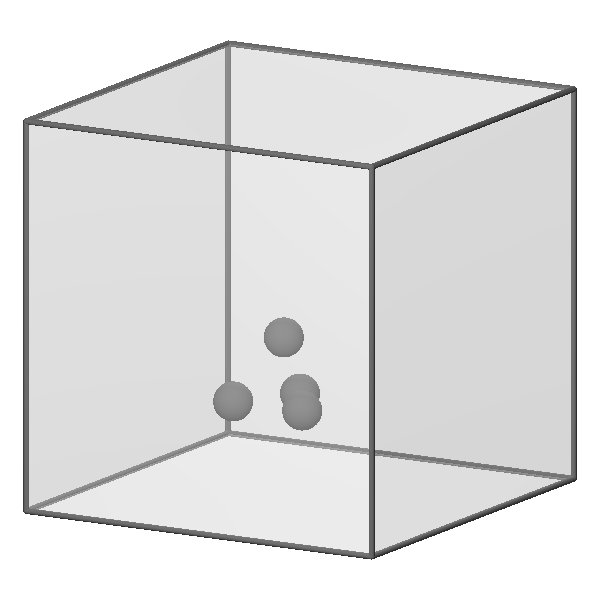}\quad%
 \raise.5in\hbox{$\S_2\Lambda^3$}\includegraphics[width=1in]{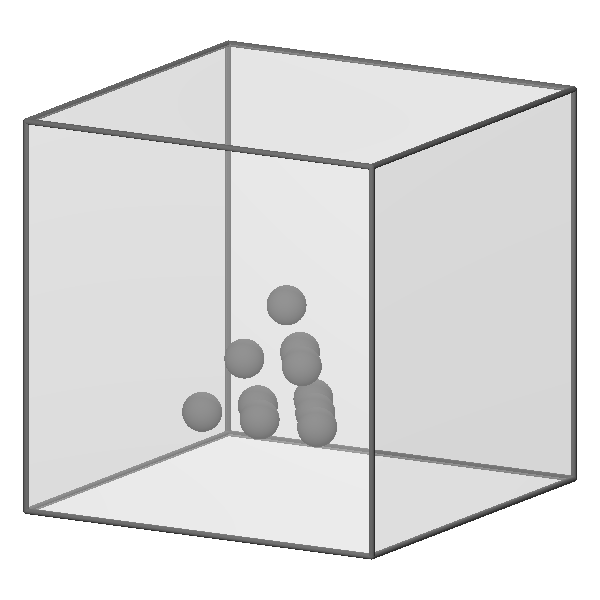}\quad%
 \raise.5in\hbox{$\S_3\Lambda^3$}\includegraphics[width=1in]{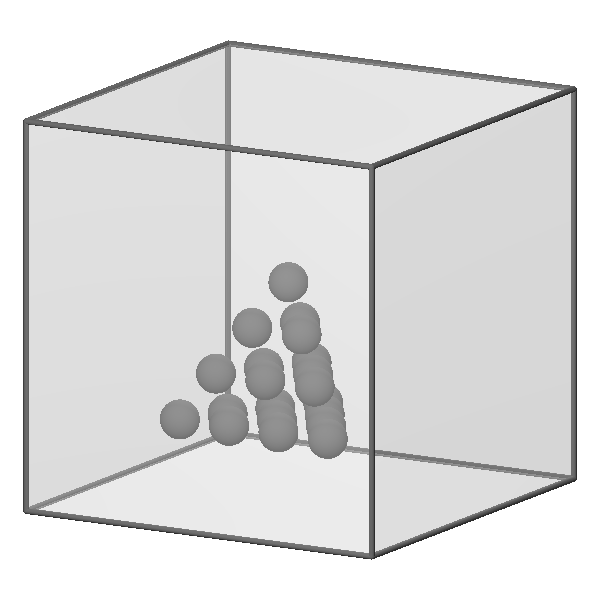}%
}
\leftline{Legend: symbols represent the value or moment of the indicated quantities:}
\leftline{\raise.25pt\hbox{\includegraphics[height=6pt]{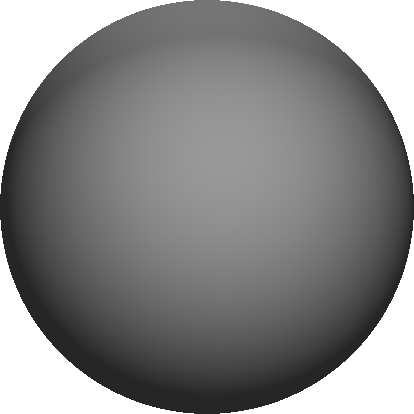}}~--~scalar (1 DOF); \quad
\raise.25pt\hbox{\includegraphics[height=6pt]{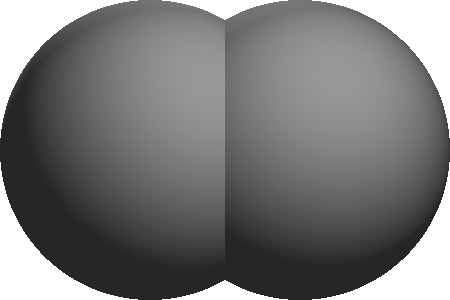}}~--~tangential vector field (2 DOFs);}
\leftline{\raise.25pt\hbox{\includegraphics[height=6pt]{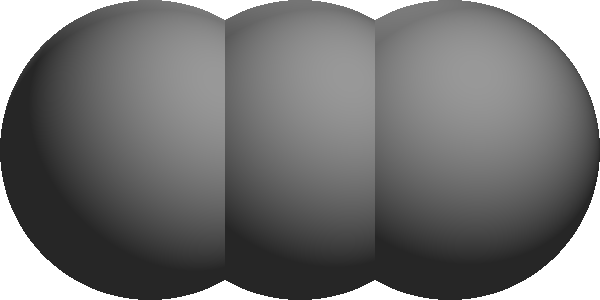}}~--~vector field (3 DOFs); \quad
\raise.25pt\hbox{\includegraphics[height=6pt]{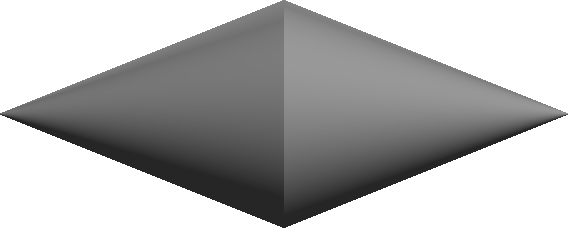}}~--~tangential
component of vector field on edge}
\leftline{or normal component on face (1 DOF).}
\caption{Element diagrams for $\S_r\Lambda^k$ in three dimensions for $r\le 3$.}\label{elts}
\end{figure}

\begin{table}
\caption{Dimension of $\S_r\Lambda^k(I^n)$.}\label{dims}
\begin{tabular}{r|rrrrrrrr}
 \multicolumn{8}{c}{$r$}\\
$k$ &  \qquad 1 & \qquad 2 & \qquad 3 & \qquad 4 & \qquad 5 & \qquad 6 & \qquad 7 \\
\hline
\\[1ex]
\hline
 \llap{$n=1$\qquad}  0 & 2 & 3 & 4 & 5 & 6 & 7 & 8 \\
 1 & 2 & 3 & 4 & 5 & 6 & 7 & 8 \\
\hline
\\[1ex]
\hline
 \llap{$n=2$\qquad} 0 & 4 & 8 & 12 & 17 & 23 & 30 & 38 \\
 1 & 8 & 14 & 22 & 32 & 44 & 58 & 74 \\
 2 & 3 & 6 & 10 & 15 & 21 & 28 & 36\\
\hline
\\[1ex]
\hline
 \llap{$n=3$\qquad} 0 & 8 & 20  & 32 & 50 & 74 & 105 & 144 \\
 1 & 24 & 48 & 84 & 135 & 204 & 294 & 408 \\
 2 & 18 & 39 & 72 & 120 & 186 & 273 & 384 \\
 3 & 4 & 10 & 20 & 35 & 56 & 84 & 120\\
\hline
\\[1ex]
\hline
 \llap{$n=4$\qquad} 0 & 16 & 48 & 80 & 136 & 216 & 328 & 480 \\
 1 & 64 & 144 & 272 & 472 & 768 & 1188 & 1764 \\
 2 & 72 & 168 & 336 & 606 & 1014 & 1602 & 2418 \\
 3 & 32 & 84 & 180 & 340 & 588 & 952 & 1464 \\
 4 & 5 & 15 & 35 & 70 & 126 & 210 & 330
\\
\hline
\multicolumn{2}{c}{}
\end{tabular}
\end{table}

Since the development of the first stable mixed finite elements for the Poisson equation
by Raviart and Thomas in 1977, such elements have proven to be
powerful tools for numerical computation.  Their paper \cite{Raviart-Thomas} introduces
a family of finite element discretizations of the space $H(\div,\Omega)$ for a two-dimensional domain
$\Omega$, one for each polynomial degree.
Together with a corresponding discontinuous piecewise polynomial discretization of $L^2(\Omega)$,
these RT spaces
stably discretize the mixed variational formulation of the Poisson equation on $\Omega$.
In \cite{Raviart-Thomas}, versions of the RT elements were given both for meshes of $\Omega$ by triangles and by rectangles.
Besides the original application of mixed finite elements for the Poisson equation, the RT elements
can be used together with the standard Lagrange finite element discretization of $H^1(\Omega)$
to give a stable mixed finite element discretization of the vector Poisson equation
$\curl\curl u - \grad\div u = f$, in which the vector variable $u$ is sought in $H(\div)$
and approximated by RT elements, and the scalar variable $\sigma = \curl u$ is sought in $H^1$
and approximated by Lagrange elements.
The RT elements were generalized to three dimensions by N\'ed\'elec \cite{Nedelec1}, with separate
generalizations giving discretizations of $H(\div,\Omega)$ and of $H(\curl,\Omega)$, including both
tetrahedral and cubic mesh generalizations of each.

In \cite{Brezzi-Douglas-Marini}, Brezzi, Douglas, and
Marini introduced a second family of finite element discretizations of $H(\div,\Omega)$ in two dimensions,
for both triangular and rectangular meshes.  These BDM elements have also proven to be very useful.
N\'ed\'elec \cite{Nedelec2} generalized the BDM family to tetrahedral meshes in three dimensions, giving
analogues both for $H(\div)$ and $H(\curl)$.  He also defined $H(\div)$ and $H(\curl)$ elements 
on cubic meshes in \cite{Nedelec2}.  However, these cannot really be considered analogues of the 
BDM elements, as they do not seem to lead to stable mixed finite element pairs.  The generalization
of the BDM elements to $H(\div)$  on three-dimensional domains (but not $H(\curl)$),
was also made by Brezzi, Douglas, Dur\'an, and Fortin in
\cite{Brezzi-Douglas-Duran-Fortin}.\footnote{The finite element discretization of $H(\div)$ in \cite{Brezzi-Douglas-Duran-Fortin}
is identical with that of \cite{Nedelec2}
in the case of tetrahedral meshes.  However, the interior degrees of freedom differ.  This does not
affect the stability analysis for the mixed Poisson equation, but only the ones given by \cite{Nedelec2}
can be used to establish stability when discretizing a mixed formulation
of the vector Laplacian.}  The paper \cite{Brezzi-Douglas-Duran-Fortin} also introduced an analogue of the
rectangular BDM elements to cubic meshes in three dimensions.

The finite element exterior calculus \cite{acta,bulletin} has greatly clarified the relation of many of
these mixed finite element methods.  
In exterior calculus the space $H(\div)$ is viewed as a space
of differential forms of degree $n-1$ in $n$ dimensions, while $H(\curl)$ is a space of $1$-forms,
$H^1$ a space of $0$-forms, and $L^2$ a space of $n$-forms.  These spaces are connected via the de~Rham
complex, in which the fundamental differential operators  $\grad$, $\curl$, $\div$ (and others in
higher dimensions) are unified as the exterior derivative.
We refer the reader to Table~2.2 in \cite{acta} for a summary of
correspondences between differential forms
and vector fields.  In \cite{acta,bulletin} two fundamental
families of finite element differential forms are defined on simplicial meshes in $n$ dimensions.
Tables~5.1 and 5.2 in \cite{acta} summarize the correspondences between these spaces
of finite element differential forms and
classical finite element spaces in two and three dimensions. 
The $\P_r^-\Lambda^k$ family specializes to the Lagrange elements, the RT elements, and the fully discontinuous
polynomial elements, for $k=0$, $1$, and $2$, respectively, in two dimensions, and to the Lagrange elements,
N\'ed\'elec's generalizations of the RT elements to $H(\curl)$ and to $H(\div)$, and the discontinuous
elements for $k=0,\ldots,3$ in three dimensions.  Taken together, these spaces form a complex
which is a finite element subcomplex of the de~Rham complex:
$$
0\to \P_r^-\Lambda^0 \xrightarrow{d}\P_r^-\Lambda^1  \xrightarrow{d}\cdots\xrightarrow{d}\P_r^-\Lambda^n\to 0.
$$
The second family of finite elements discussed in
\cite{acta,bulletin} is the $\P_r\Lambda^k$ family.  For $0$-forms and $n$-forms this
brings nothing new, just recapturing the Lagrange and fully discontinuous polynomial elements,
but for $0<k<n$, this is a different family of finite element spaces.  For $n=2$, $k=1$
it gives the BDM triangular elements, and for $n=3$, $k=1$ and $2$, it gives N\'ed\'elec's
generalizations of these to $H(\div)$ and $H(\curl)$.  They combine into a second finite element
de~Rham subcomplex
$$
0\to \P_r\Lambda^0 \xrightarrow{d}\P_{r-1}\Lambda^1  \xrightarrow{d}\cdots\xrightarrow{d}\P_{r-n}\Lambda^n\to 0.
$$
Note that the degree decreases in this complex, in contrast to the preceeding one.

We now turn to the construction of analogous spaces and complexes for cubical meshes.
An analogue to the $\P_r^-\Lambda^k$ complex of elements for cubical meshes may be easily
constructed via a tensor product construction.  For an explicit description in $n$ dimensions
and for all form degrees $k$, see \cite{arnold-boffi-bonizzoni}.  This includes the tensor product
Lagrange, or $\Q_r$, elements for $0$-forms, the rectangular
RT elements for $1$-forms in $2$-D, and the 3-D generalizations of them given in \cite{Nedelec1}.
The space for $n$-forms ($L^2$) is the fully discontinuous space of tensor product polynomials
(shape functions in $\Q_r$).

This paper develops a second family of finite element spaces for cubical meshes.  This family
may be viewed as an analogue of the $\P_r\Lambda^k$ family for cubical meshes.  We denote
the new family of elements by $\S_r\Lambda^k$ and define such a space for
all dimensions $n\ge 1$, all polynomial degrees $r\ge 1$, and all form degrees
$0\le k\le n.$  For $0$-forms
(discretization of $H^1$), $S_r\Lambda^0$ is not equal to $\Q_r$, but rather to the serendipity
elements in 2-D, their three-dimensional extension (which can be found in many places
for small values of $r$ and for general $r$ in \cite{MR1164869} and \cite{lin-zhang}),
and to a recent extension to all dimensions \cite{serendipity}.
For $n$-forms, $\S_r\Lambda^n$ uses fully discontinuous elements with shape functions in $\P_r$
(not $\Q_r$).

For $1$-forms in 2-D, the space $\S_r\Lambda^1$ coincides with the rectangular BDM elements
of \cite{Brezzi-Douglas-Marini}.  For $2$-forms
in 3-D, we believe that the space $\S_r\Lambda^2$, which has not appeared
before as far as we know, is the correct analogue of the BDM elements on
cubic meshes.  It has the same degrees of freedom as the space given in \cite{Brezzi-Douglas-Duran-Fortin}
but the shape functions have better symmetry properties.  For $1$-forms in 3-D,
$\S_r\Lambda^1$ is a finite element discretization of $H(\curl)$.  To the best of our knowledge,
neither the degrees of freedom nor the shape functions for this space
have been proposed previously.  Even for $0$-forms, the spaces $\S_r\Lambda^0$ were
only discovered very recently
in higher dimensions.  They are the generalization of the
serendipity
spaces given by the present authors in \cite{serendipity}.  That work was motivated
by the search for a finite element discretization of the de~Rham complex on cubical meshes
which is completed in this paper.

In the remainder of the paper we will develop the $\S_r\Lambda^k$ spaces and their properties
in the setting of differential forms on cubical meshes of arbitrary dimensions.
Here we explicitly describe the spaces which arise in three dimensions in traditional finite element
terminology by giving their shape functions and degrees of freedom.  See also Figure~\ref{elts}.

\emph{The space $\S_r\Lambda^3$.}
This is simply the space of piecewise polynomials of degree at most $r$
with no continuity requirements.  Obviously this gives a finite element subspace of $L^2$
(which, unlike the remaining spaces, is defined also for $r=0$).  Its dimension is
$(r+1)(r+2)(r+3)/6$.

\emph{The space $\S_r\Lambda^2$.}
The shape functions for the space $\S_r\Lambda^2$ on the unit cube
$I^3$ in three dimensions are vector polynomials
of the form
\begin{equation}\label{newbdm}
u = (v_1,v_2,v_3)  + \curl (x_2x_3(w_2-w_3),x_3x_1(w_3-w_1),x_1x_2(w_1-w_2)),
\end{equation}
with $v_i, w_i\in \P_r(I^3)$ and  $w_i$ independent of $x_i$.
The dimension of this space is $(r+1)(r^2+5r+12)/2$.  As degrees of freedom for
$u\in \S_r\Lambda^2$ we take
$$
u\mapsto \int_f u\cdot n\,p, \ p\in \P_r(f), \  f \text{ a face of $I^3$}; \quad
 u\mapsto \int_{I^3} u\cdot p, \ p\in [\P_{r-2}(I^3)]^3.
$$
These degrees of freedom 

\emph{The space $\S_r\Lambda^1$.}
The shape functions for the space $\S_r\Lambda^1$
in three dimensions are vector polynomials
of the form
\begin{equation}\label{newcurl}
u = (v_1,v_2,v_3)  + (x_2x_3(w_2-w_3),x_3x_1(w_3-w_1),x_1x_2(w_1-w_2)) +\grad s,
\end{equation}
with $v_i\in \P_r(I^3)$, $w_i\in \P_{r-1}(I^3)$ independent of $x_i$, and
$s$ is a polynomial of $I^3$ with superlinear degree at most $r+1$,
where the superlinear degree of a polynomial is the ordinary degree ignoring
variables which enter linearly (e.g., the superlinear degree of $x_1^2 x_2 x_3^3$ is $5$).
The dimension of this space is $(r+1)(r^2+5r+18)/2$.  As degrees of freedom for $u\in \S_r\Lambda^1$ 
we take
\begin{gather*}
u\mapsto \int_e u\cdot t\,p, \ p\in \P_r(e), \  e \text{ an edge of $I^3$};
\\
u\mapsto \int_f u\x n\,p, \ p\in [\P_{r-2}(f)]^2, \  f \text{ a face of $I^3$}; \quad
 u\mapsto \int_{I^3} u\cdot p, \ p\in [\P_{r-4}(I^3)]^3.
\end{gather*}

\emph{The space $\S_r\Lambda^0$.}
Finally, the space $S_r\Lambda^0$ is the generalized serendipity space of \cite{serendipity}.
The shape functions are all polynomials of superlinear degree at most $r$, and the degrees
of freedom are the values at the vertices and the moments of degree at most $r-2$, $r-4$, and
$r-6$ on the edges, faces, and interior, respectively.  The dimension of this space is
$8$ for $r=1$, $20$ for $r=2$, and $(r+1)(r^2+5r+24)/6$ for $r\ge 3$.

As a consequence of the general theory below, the degrees of freedom given are unisolvent
for all these spaces,
and for any cubical decomposition $\T$ of $\Omega$,
the assembled finite element spaces $\S_r\Lambda^k(\T)$ have exactly the continuity required to belong to
$H^1(\Omega)$, $H(\div,\Omega)$, $H(\curl,\Omega)$, and $L^2(\Omega)$, for $k=0,\ldots,3$, respectively.
In other words, $\S_r\Lambda^k(\T)$ belongs to the domain of the exterior derivative on $k$-forms.
Moreover, exterior derivative maps $\S_r\Lambda^k(\T)$ into $\S_{r-1}\Lambda^{k+1}(\T)$, so
we obtain a finite element subcomplex
$$
0\to \S_r\Lambda^0(\T)\xrightarrow{\grad} \S_{r-1}\Lambda^1(\T)\xrightarrow{\curl}
\S_{r-2}\Lambda^2(\T)\xrightarrow{\div}\S_{r-3}\Lambda^3(\T)\to 0
$$
of the de~Rham complex.
Finally, if we define projection operators $\pi^k$ from smooth fields into
the finite element spaces $S_r\Lambda^k$ using the degrees of freedom,
these commute with exterior differentiation.  That is,
the following diagram commutes:
$$
\begin{CD}
0 @>>> C^\infty(\Omega) @>\grad>> [C^\infty(\Omega)]^3  @>\curl>> [C^\infty(\Omega)]^3  @>\div>> C^\infty(\Omega) @>>>0
\\
@. @V\pi^0VV  @V\pi^1VV @V\pi^2VV  @V\pi^3VV
\\
0 @>>> \S_r\Lambda^0(\T) @>\grad>> \S_{r-1}\Lambda^1(\T)  @>\curl>> \S_{r-2}\Lambda^2(\T) @>\div>> \S_{r-3}\Lambda^3(\T) @>>>0
\end{CD}
$$
where, for simplicity, we assume $r\ge 3$ (otherwise some of the spaces are undefined and some
parts of the diagram are not applicable).

The remainder of the paper consists of four sections.  In Section~2, we recall some key concepts from
exterior calculus, particularly the Koszul differential and Koszul complex of polynomial differential forms,
which will be crucial to our construction.
We introduce the concept of linear degree, and show that the subcomplex obtained from the Koszul complex
by placing a lower bound on linear degree is exact.  This is a key step in the unisolvence proof
in Section~3.  In Section~3, we define the spaces $\S_r\Lambda^k$ on an
$n$-dimensional cube by giving the
shape functions and degrees of freedom.  We derive a number of properties from these definitions,
leading to a formula for the dimension of the space of shape functions and a proof of unisolvence
of the degrees of freedom.  
We then define projection operators $\pi_r^k$ mapping smooth fields into the finite element spaces $\S_r \Lambda^k$
and show that the projections commute with the exterior derivative.
This accomplished, the new elements fit squarely into the framework of the finite element
exterior calculus given in \cite{bulletin}.  Therefore the application of the elements
to PDE problems and their numerical analysis does not require
new ideas, and so we do not discuss that here.
The proof of unisolvence in Section~3 hinges on the special case of functions with vanishing trace.
The proof in that case is the topic of Section~4.

\section{Notation and preliminaries}
For $n\ge1$, the number of dimensions, let $\N_n=\{1,\ldots,n\}$,
and let $\Sigma(k)$, $0\le k\le n$, denote the set of subsets of $\N_n$ consisting of $k$ elements.
For $\sigma\in\Sigma(k)$ we denote by $\sigma^*\in\Sigma(n-k)$ its complement $\N_n\setminus\sigma$.
For $p\in\sigma$ we write $\sigma-p$ for $\sigma\setminus\{p\}\in\Sigma(k-1)$,
and for $q\in\sigma^*$ we write $\sigma+q$ for $\sigma\cup\{q\}\in\Sigma(k+1)$.
If $\sigma\subset\N_n$ and $q\in\sigma^*$ we let
$\epsilon(q,\sigma)=(-1)^l$
where $l=\#\{p\in\sigma|p<q\}$, and set 
$\epsilon(q,p,\sigma):=\epsilon(q,\sigma)\epsilon(p,\sigma+q-p)$ for $q\in\sigma^*$, $p\in\sigma$.
For later reference, we note that
\begin{equation}\label{epseps}
 \epsilon(q,p,\sigma)=-\epsilon(p,\sigma-p)\epsilon(q,\sigma-p),
\end{equation}
which is easily verified by considering the cases $q<p$ and $q>p$ separately.

We now recall some basic tools and results of exterior algebra and exterior calculus.
These can be found, for example, in \cite[Section~2]{acta}.
For each $\sigma\in\Sigma(k)$ we denote by $\sigma_1,\ldots,\sigma_k$ its elements
in increasing order, and by
$$
dx_\sigma=dx_{\sigma_1}\wedge\cdots\wedge dx_{\sigma_k}
$$
the corresponding basic alternator.
A differential $k$-form $\omega$ on a domain $\Omega\subset\R^n$ may be written as
\begin{equation}\label{kform}
\omega = \sum_{\sigma\in\Sigma(k)} \omega_\sigma\,dx_\sigma,
\end{equation}
with coefficients $\omega_\sigma$ belonging to any desired space of functions
on $\Omega$, e.g., $L^2(\Omega)$.
A $0$-form is simply such a function.  The exterior derivative of the differential form
\eqref{kform} is
$$
d\omega = \sum_{\sigma\in\Sigma(k)} \sum_{q\in\N_n} \partial_q\omega_\sigma\,dx_q\wedge dx_\sigma,
$$
where $\partial_q=\partial/\partial x_q$.

A differential $k$-form may be
contracted with a vector field on $\Omega$ to give a differential $(k-1)$-form
(or zero if $k=0$).  When the vector field is simply the identity, the resulting
operator is the Koszul differential.  
Equivalently, we may define the Koszul differential
on the basic alternators by 
$$
\kappa(dx_{\sigma_1}\wedge\cdots\wedge dx_{\sigma_k}) =  \sum_{i=1}^k (-1)^{i+1} x_{\sigma_i}\,
dx_{\sigma_1}\wedge\cdots\wedge\widehat{dx_{\sigma_i}}\wedge\cdots\wedge dx_{\sigma_k},
$$
and then extend it to a general differential form
by linearity:
\begin{equation}\label{defk}
\kappa(\sum_\sigma \omega_\sigma\,dx_\sigma) =  \sum_\sigma \omega_\sigma\, \sum_{i=1}^k (-1)^{i+1} x_{\sigma_i}\,
dx_{\sigma_1}\wedge\cdots\wedge\widehat{dx_{\sigma_i}}\wedge\cdots\wedge dx_{\sigma_k}.
\end{equation}
The operator $\kappa$ is a graded differential, meaning that
$$
\kappa(\kappa\omega) =0,\quad
\kappa(\omega\wedge\eta) = \kappa\omega\wedge\eta + (-1)^k\omega\wedge\kappa\eta,
$$
if $\omega$ is a $k$-form and $\eta$ an $l$-form.

The following lemma collects formulas for $\kappa\omega$, $d\omega$, and $\kappa d\omega$.
\begin{lemma}\label{kdkd} If $\omega$ is given by \eqref{kform}, then
\begin{gather}\label{kformula}
\kappa\omega = \sum_{\zeta\in\Sigma(k-1)}\eta_\zeta\,dx_\zeta,
\text{ where }
\eta_\zeta = \sum_{q\in\zeta^*}\epsilon(q,\zeta) x_q \omega_{\zeta+q},
\\\label{dformula}
d\omega = \sum_{\rho\in\Sigma(k+1)}\nu_\rho\,dx_\rho,
\text{ where }
\nu_\rho = \sum_{q\in\rho}\epsilon(q,\rho-q) \partial_q \omega_{\rho-q},
\\\label{kdformula}
\kappa d\omega = \sum_{\sigma\in\Sigma(k)}\mu_\sigma\,dx_\sigma,
\text{ where }
\mu_\sigma = \sum_{q\in\sigma^*}\bigl[x_q\partial_q\omega_\sigma +
\sum_{p\in\sigma}\epsilon(q,p,\sigma)x_q\partial_p\omega_{\sigma+q-p}\bigr].
\end{gather}
\end{lemma}
\begin{proof}
 By definition
$$
\kappa dx_\sigma = \sum_{q\in\sigma}\epsilon(q,\sigma-q)x_q\, dx_{\sigma-q},
$$
so
$$
\kappa \omega = \sum_{\sigma\in\Sigma(k)}\sum_{q\in\sigma}
  \epsilon(q,\sigma-q)x_q\omega_\sigma\, dx_{\sigma-q}.
$$
Making the change of variable $\zeta  = \sigma-q$, so
$$
\sigma\in\Sigma(k), \quad q\in\sigma \quad\iff\quad \zeta\in\Sigma(k-1),\quad q\in\zeta^*,
$$
we obtain \eqref{kformula}.  The second result is proven similarly, and the third follows from
the first two.
\end{proof}

We now turn to differential forms with polynomial coefficients.
A monomial in $n$ variables is determined by
a multi-index $\alpha$ of $n$ nonnegative integers: $x^\alpha = x_1^{\alpha_1}\cdots x_n^{\alpha_n}$.
By a form monomial in $n$ variables, we mean the product of a monomial with a basic alternator:
$m = x^\alpha \,dx_\sigma$ for some multi-index $\alpha$ and $\sigma\subset\N_n$.
The polynomial degree and the linear degree of $m$ are defined as
$$
\deg m = |\alpha|:=\sum_i \alpha_i, \quad
\ldeg m = \# \{\,i\in\sigma^*\,|\,\alpha_i=1\}.
$$
Thus the linear degree of $m$ is the degree of its polynomial coefficient counting only those variables
which enter linearly, and excluding variables which enter the alternator.  For ordinary monomials,
i.e., 0-forms, the linear degree is equal to the difference between the polynomial degree
and the superlinear degree which appeared in the introduction.

We define $\H_r\Lambda^k=\H_r\Lambda^k(\R^n)$ to be the span of the $k$-form monomials $m$ with $\deg m = r$,
and $\P_r\Lambda^k=\P_r\Lambda^k(\R^n)=\sum_{s=0}^r \H_s\Lambda^k$
to be the span of those with $\deg m \le r$.  If $k=0$, we may simply write $\H_r$ and $\P_r$.
If $\Omega$ is a subdomain of $\R^n$, we define $\P_r\Lambda^k(\Omega)$
to be the space of restrictions to $\Omega$ of the elements of $\P_r\Lambda^k$ (and similarly for other
function spaces).  Note that $d$ maps $\H_r\Lambda^k$ into $\H_{r-1}\Lambda^{k+1}$ while
$\kappa$ maps $\H_r\Lambda^k$ into $\H_{r+1}\Lambda^{k-1}$.
An extremely useful identity is the \emph{homotopy formula} (\cite[Theorem~3.1]{acta}):
\begin{equation}\label{hmt}
(d\kappa + \kappa d)\omega = (r+k)\omega, \quad \omega\in\H_r\Lambda^k(\R^n). 
\end{equation}

We extend the linear degree for form monomials to
polynomial differential forms by defining $\ldeg\mu$ for any  $\mu\in\H_r\Lambda^k$
to be the \emph{minimum} of the
linear degree among all the monomials $m$ in
$\mu$.  We say that $\mu$ is of homogeneous linear degree equal to $l$ if $\ldeg m=l$ for every monomial
of $\mu$.  We denote by $\H_{r,l}\Lambda^k$ the space of forms in $\H_r\Lambda^k$ of linear degree at least $l$.
Obviously,
\begin{equation}\label{lmax}
\H_{r,l}\Lambda^k(\R^n)=0, \quad l>\min(r,n-k).
\end{equation}

The exterior derivative $d$ may decrease the linear degree of a polynomial differential form,
but $\kappa$ and $\kappa d\kappa$ do not.
\begin{lemma}\label{degs}
For any $\omega\in\P_r\Lambda^k$, $\ldeg\kappa\omega\ge\ldeg\omega$ and $\ldeg\kappa d\kappa\omega\ge\ldeg\omega$.
\end{lemma}
\begin{proof}
 If $m$ is monomial of $\omega$ and $l=\ldeg m$, then it follows directly from
the definition \eqref{defk} that the monomials
of $\kappa m$ are of linear degree $l$ and/or $l+1$, so $\ldeg \kappa m\ge l\ge \ldeg\omega$.
Since every monomial of $\kappa\omega$ is a monomial of $\kappa m$ for some
monomial $m$ of $\omega$, this implies the first inequality.
For the second we use the differential property of $\kappa$ and the homotopy formula to
see that
$\kappa d\kappa m = (\kappa d+d\kappa)\kappa m$ is a multiple of $\kappa m$.  
Therefore $\ldeg \kappa d\kappa m = \ldeg\kappa m\ge\ldeg\omega$, which
gives the second inequality.
\end{proof}
In view of Lemma~\ref{degs}, for each $l\ge0$, we obtain a complex
\begin{equation}\label{lcomplex}
\cdots
\xrightarrow{\kappa}\H_{r-1,l}\Lambda^{k+1}\xrightarrow{\kappa} \H_{r,l}\Lambda^k\xrightarrow{\kappa}
\cdots.
\end{equation}
When $l=0$ this is the Koszul complex, and exactness follows from the homotopy formula.
In fact, the complex \eqref{lcomplex} is exact for all $l\ge0$.

\begin{thm}\label{t1}
For $r\ge 1$, $0\le l < r$, $0\le k< n$, the sequence
$$
\H_{r-1,l}\Lambda^{k+1}\xrightarrow{\kappa} \H_{r,l}\Lambda^k\xrightarrow{\kappa} \H_{r+1,l}\Lambda^{k-1}
$$
is exact.  Equivalently,
$\dim \kappa(\H_{r-1,l}\Lambda^{k+1}) + \dim\kappa(\H_{r,l}\Lambda^k)=
\dim \H_{r,l}\Lambda^k$.
\end{thm}

The proof is due to Scot Adams and Victor Reiner \cite{adams-reiner}.  Its
main ingredient is contained in the following lemma.
\begin{lemma}\label{l1}
Let $r\ge 1$, $l\ge 1$, and $0\le k < n$.
Suppose that $\mu\in \H_{r}\Lambda^{k}$ is of linear degree at least $l-1$
and $\kappa\mu$ is of linear degree at least $l$.
Then there exists $\nu\in\H_{r-1}\Lambda^{k+1}$ such that $\mu-\kappa\nu$ is of
linear degree at least $l$.  Further, if $\mu\in \H_r\Lambda^n$ is nonzero, then $\kappa\mu$
is of linear degree $0$.
\end{lemma}
\begin{proof}
For the final statement, concerning $n$-forms, we write $\mu = p\,dx_1\wedge\cdots\wedge dx_n$
where $p\in \H_r$.  Since $r\ge 1$, $p$ is not constant.  But then
$\kappa\mu = \sum_i(-1)^{i+1}p x_i\, dx_1\wedge\cdots\wedge\widehat{dx_i}\wedge\cdots\wedge dx_n$
is easily seen to be of linear degree $0$.

For $0\le k< n$, the proof hinges on
a canonical form for an element of $\mu\in\H_r\Lambda^k$ which we establish before proceeding.
Let us say that a form monomial $x^\alpha\,dx_\sigma$ is \emph{full} if $\sigma\subset\supp(\alpha)$, the support of $\alpha$.
To each of the monomials $m=x^\alpha\,dx_\sigma$
of $\mu\in \H_r\Lambda^k$,
we associate the increasing sequences  $\rho$ and $\tau$ with $\rho=\sigma\cap \supp(\alpha)$ and
$\tau=\sigma\setminus\supp(\alpha)$.  Then
$m=\pm\eta\wedge dx_{\tau}$
where $\eta= x^\alpha dx_{\rho}$ is a full form monomial which
is independent of the $\tau$ variables (that is,
$\supp(\alpha)\cup\rho$ is disjoint from $\tau$).  Note that
$\ldeg\eta=\ldeg m$.
Finally, in the expansion of $\mu$ as a linear
combination of its monomials, we gather together the terms with the
same $\tau=\sigma\setminus\supp(\alpha)$, and in this way write
\begin{equation}\label{cf}
 \mu = \sum_{\tau\subset\N_n} \eta_\tau \wedge dx_{\tau},
\end{equation}
where $\eta_\tau\in \H_r\Lambda^{k-\#\tau}$ is independent of the $\tau$ variables, and has all
of its monomials full.  The expression on the right-hand side of \eqref{cf} is
the desired canonical form of $\mu$.

Now we proceed with the proof of the lemma.  We consider first the special case in
which $\mu$ is of homogeneous linear degree $l-1$, and, in this special case, we
use induction on $l$, the case $l=0$ being known (exactness of the Koszul complex).
Expressing $\mu$ in the canonical
form \eqref{cf}, we have $\ldeg\eta_\tau=l-1$ for each $\tau$ (for which the coefficient $\eta_\tau$
does not vanish).  Now
\begin{equation}\label{kcf}
\kappa \mu = \sum \kappa\eta_\tau \wedge dx_{\tau} + \sum \pm \eta_\tau \wedge \kappa dx_{\tau}.
\end{equation}
The first sum is of homogeneous linear degree $l-1$
and the second of homogeneous linear degree $l$.  Since we assumed that $\kappa\mu$ is of linear
degree at least $l$, the first sum must vanish.
But this sum is in canonical form, so we conclude that $\kappa\eta_\tau=0$
for each $\tau$.  Invoking the inductive hypothesis, we can write
$\eta_\tau=\kappa \nu_\tau$ where $\nu_\tau$ has linear degree at least $l-1$.
Let
$$
\nu = \sum  \nu_\tau \wedge dx_{\tau}\in \H_{r-1}\Lambda^{k+1}.
$$
Then
$$
\mu-\kappa\nu = \sum \eta_\tau \wedge dx_{\tau} -  \sum \kappa\nu_\tau\wedge dx_{\tau}
  - \sum \pm\nu_\tau\wedge \kappa dx_{\tau}
= - \sum \pm\nu_\tau\wedge \kappa dx_{\tau}.
$$
Expanding the right-hand side into monomials, we see that each has $\ldeg$ at least $l$, so this
completes the proof under the assumption that $\mu$ is of homogeneous linear degree $l-1$.

Next we turn to the general case, in which $\mu$ is of linear degree at least $l-1$.
We may split $\mu$ as $\mu'+\mu''$ with $\mu'$ of homogeneous linear degree $l-1$
and $\mu''$ of linear degree at least $l$.
Then $\kappa\mu'$ splits into a part of homogeneous linear degree $l-1$ and a part
of homogeneous linear degree $l$, while $\ldeg(\kappa\mu'')\ge l$. 
Since, by assumption, $\ldeg(\kappa\mu)\ge l$, 
the part of $\kappa\mu'$ with linear degree equal to $l-1$ must vanish.  That is, $\ldeg(\kappa\mu')\ge l$.
Therefore, we may apply the result of the preceding special case
to $\mu'$ to obtain $\nu\in\H_{r-1}\Lambda^{k+1}$ such that $\ldeg(\mu'-\kappa\nu)\ge l$.
Then $\mu-\kappa\nu=(\mu'-\kappa\nu) + \mu''$ is of linear degree at least $l$. 
\end{proof}

Finally, we give the proof of Theorem~\ref{t1}.
\begin{proof}
The result is certainly true for $l=0$, so we may assume $1\le l < r$ (and so $r\ge 2$).
Suppose $\omega\in \H_{r,l}\Lambda^k$ with $\kappa\omega=0$.  We must show that there
exists $\eta\in \H_{r-1}\Lambda^{k+1}$ with linear degree at least $l$, such that $\kappa\eta=\omega$.
Now $\mu:=d\omega/(r+k)\in \H_{r-1}\Lambda^{k+1}$ is of linear degree at least $l-1$
and satisfies
$\kappa\mu = (\kappa d + d\kappa)\omega/(r+k)=\omega$ by \eqref{hmt}.  If $k=n-1$, the final sentence
of the lemma insures that $\mu=0$.  For $0\le k <n-1$,
we apply the lemma with $r$ and $k$ replaced by $r-1$ and $k+1$, respectively, and conclude that
there exists $\nu\in\H_{r-2}\Lambda^{k+2}$ such that $\eta:=\mu-\kappa\nu$ is of linear
degree at least $l$.
Clearly, $\kappa\eta=\kappa\mu = \omega$.
\end{proof}

\section{The $\S_r\Lambda^k$ spaces}
Here, the main section of the paper, we define the polynomial spaces
$\S_r\Lambda^k(\R^n)$ we shall use as shape functions (see \eqref{defS} below) and the degrees of
freedom for these (see \eqref{dofs}).  We derive a number of properties of these
polynomial spaces in Theorems~\ref{degprop} through \ref{trace}
and use them to verify unisolvence in Theorem~\ref{unisolv}.

The space of shape functions will consists of polynomials of
a given degree plus certain additional terms of higher degree which
will be defined in terms of the following auxilliary space:
$$
\J_r\Lambda^k(\R^n) =  \sum_{l\ge 1}\kappa\, \H_{r+l-1,l}\Lambda^{k+1}(\R^n).
$$
In view of \eqref{lmax}, the sum is finite and
\begin{equation}\label{jdeg}
\J_r\Lambda^k(\R^n)\subset \P_{r+n-k-1}\Lambda^k(\R^n).
\end{equation}
Moreover, the sum is direct,
since the polynomial degrees of the summands differ.  The following proposition,
which follows directly from the definitions,
helps to clarify the meaning of this space.
\begin{prop}\label{jspace}
{\rm 1.} The space $\sum_{l\ge 1}\H_{r+l-1,l}\Lambda^{k}(\R^n)$ is the span of
all $k$-form monomials $m$ with $\deg m \ge r$ and
$\deg m-\ldeg m \le r-1$.

{\rm 2.} The space $\J_r\Lambda^k(\R^n)$ is
the span of $\kappa m$ for all $(k+1)$-form monomials $m$ with $\deg m \ge r$ and
$\deg m-\ldeg m \le r-1$.
\end{prop}

For several values of
$k$,
this space can be described more explicitly.  By \eqref{lmax},
\begin{equation}\label{jvan}
\J_r\Lambda^k(\R^n) = 0, \text{ for $k=n$ or $n-1$},
\end{equation}
while
$$
\J_r\Lambda^{n-2}(\R^n) =  \kappa\, \H_{r,1}\Lambda^{n-1}(\R^n).
$$
Now, $\H_{r,1}\Lambda^{n-1}(\R^n)$ is the span of the monomials
$x_i w_i \theta_i$,
where $w_i\in \H_{r-1}(\R^n)$ is independent of $x_i$
and $\theta_i := (-1)^{i-1}  d x_1 \wedge \ldots 
\wedge\widehat{dx^i}\wedge \ldots \wedge d x_n$.  We then have
$$
\kappa \theta_i = \sum_{j<i}x_j\theta_{j,i} - \sum_{j>i}x_j \theta_{i,j},
$$
with
$$
\theta_{i,j}  = (-1)^{i+j} d x_1 \wedge \ldots 
\wedge\widehat{dx^i}\wedge \ldots\wedge\widehat{dx^j}\wedge \ldots \wedge d x_n.
$$
Therefore,
\begin{equation}\label{jn-2}
\J_r\Lambda^{n-2}(\R^n) =\{\,\sum_{i<j}x_ix_j(w_i-w_j)\theta_{i,j}\,|\, w_i\in \H_{r-1}(\R^n)
\text{ independent of $x_i$}\,\}.
\end{equation}
Finally, we identify $\J_r\Lambda^0(\R^n)$.  By  Theorem~\ref{t1} in the case $k=0$,
we see that $\J_r\Lambda^0(\R^n)=\sum_{l\ge 1}\H_{r+l,l}\Lambda^0(\R^n)$.
By Proposition~\ref{jspace},
this space is the span of monomials of degree $>r$ whose superlinear degree, that is, $\deg -\ldeg$,
is at most $r$.

We can now define the space of polynomial $k$-forms which we use for shape functions,
\begin{equation}\label{defS}
\S_r\Lambda^k(\R^n):= \P_r\Lambda^k(\R^n) + \J_r\Lambda^k(\R^n) +d\J_{r+1}\Lambda^{k-1}(\R^n),
\end{equation}
defined for all $r\ge 1$ and all $0\le k \le n$.  From \eqref{jdeg},
\begin{equation}\label{dbd}
 \S_r\Lambda^k(\R^n)\subset\P_{r+n-k}\Lambda^k(\R^n).
\end{equation}

Note that, in case $k=0$, the final term in \eqref{defS} vanishes, and, by the characterization of
$\J_r\Lambda^0(\R^n)$ just derived, $\S_r\Lambda^0(\R^n)$ consists precisely of the span
of all monomials of superlinear degree at most $r$.
This is exactly the serendipity space as defined in \cite{serendipity}.  In this case, \eqref{jdeg}
gives the sharper degree bound
\begin{equation}\label{sharper}
 \S_r\Lambda^0(\R^n)\subset\P_{r+n-1}\Lambda^0(\R^n).
\end{equation}

Another case in which the expression for $\S_r\Lambda^k(\R^n)$
can be simplified is when $k=n$.  By \eqref{jvan},
$$
\S_r\Lambda^n(\R^n) = \P_r\Lambda^n(\R^n).
$$
When $k=n-1$, we have, by \eqref{jvan},
$$
\S_r\Lambda^{n-1}(\R^n):= \P_r\Lambda^{n-1}(\R^n)  +d\J_{r+1}\Lambda^{n-2}(\R^n),
$$
where the last space is characterized in \eqref{jn-2}.  In the case of three dimensions,
this is formula \eqref{newbdm} given in the introduction, stated in the language
of exterior calculus.  In a similar way, we recover formula \eqref{newcurl} for
the $3$-D $H(\curl)$ elements discussed in the introduction.

We now derive several properties of these polynomial spaces.  The first limits the monomials
that appear in the polynomials in $\S_r\Lambda^k(\R^n)$.
\begin{thm}[Degree property]\label{degprop}
For any $n,r\ge 1$ and  $0\le k\le n$, the space $\S_r\Lambda^k(\R^n)$ is contained
in the span of the $k$-form monomials $m$ of degree at most $r+n-k-\delta_{k0}$
for which
\begin{equation}\label{deg}
 \deg m-\ldeg m \le r+1-\delta_{k0}.
\end{equation}
\end{thm}
\begin{proof}
The bound $r+n-k-\delta_{k0}$ on the degree is given in \eqref{dbd} for $k>0$ and
in \eqref{sharper} for $k=0$, so we need only  show 
\eqref{deg}. If $m$ is a monomial of an element of $\P_r\Lambda^k(\R^n)$, then $\deg m\le r$ and
$\ldeg m\ge 0$, so \eqref{deg} holds with $r$ on the right-hand side.
If $m$ is a monomial of an element of $\J_r\Lambda^k(\R^n)$,
then $m$ occurs in the expansion of $\kappa p$, where $p$ is a $(k+1)$-form monomial
with $\deg p-\ldeg p\le r-1$ (Proposition~\ref{jspace}).  Then
$\deg m = \deg p+1$ and, by Lemma~\ref{degs}, $\ldeg m \ge \ldeg p$, so again $\deg m-\ldeg m\le r$.  Finally,
if $k>0$ and
$m$ is a monomial of an element of $d\J_{r+1}\Lambda^{k-1}(\R^n)$, then by the argument
just given, $m$ is a monomial of $dq$ where $q$ is a $(k-1)$-form monomial with $\deg q-\ldeg q\le r+1$.
Since $\deg m=\deg q-1$ and $\ldeg m \ge \ldeg q-1$, we get \eqref{deg}.
\end{proof}

A crucial property of these polynomial form spaces, is that they can be combined
to form a subcomplex of the de~Rham complex.
\begin{thm}[Subcomplex property]\label{complex}
 Let $n,r\ge 1$, and let $0< k\le n$.  Then
$$
d\S_{r+1}\Lambda^{k-1}(\R^n) \subset \S_r\Lambda^k(\R^n).
$$
\end{thm}
\begin{proof}
With reference to \eqref{defS}, we note that
$d\P_{r+1}\Lambda^{k-1}(\R^n)\subset \P_r\Lambda^k(\R^n)$, and
$d(d\J_{r+2}\Lambda^{k}(\R^n))$ vanishes, so it suffices to prove that
$d\J_{r+1}\Lambda^{k-1}(\R^n)\subset \S_r\Lambda^k(\R^n)$,
which is immediate from \eqref{defS}.
\end{proof}

We next observe that the spaces increase with increasing polynomial degree.
\begin{thm}[Inclusion property]\label{inclusion}
 Let $n,r\ge 1$, and let $0\le k\le n$.  Then
$$
\S_{r}\Lambda^{k}(\R^n) \subset \S_{r+1}\Lambda^k(\R^n).
$$
\end{thm}
\begin{proof}
We must show that each of the three summands on the right-hand side of \eqref{defS} is
included in $\S_{r+1}\Lambda^k(\R^n)$.
Clearly,
$$
\P_r\Lambda^{k}(\R^n)\subset \P_{r+1}\Lambda^k(\R^n)\subset\S_{r+1}\Lambda^k(\R^n),
$$
which establishes the first inclusion.
Next, we show that
$$
\J_r \Lambda^k(\R^n) \subset \P_{r+1}\Lambda^k(\R^n) + \J_{r+1} \Lambda^k(\R^n)\subset\S_{r+1}\Lambda^k(\R^n).
$$
By Proposition~\ref{jspace},
elements of  $\J_r \Lambda^k(\R^n)$ are of form $\kappa m$ with $m$ a $(k+1)$-form
monomial with $\deg m \geq r, \deg m-\ldeg m\le r-1$. By the homotopy formula \eqref{hmt}, 
$\kappa m$ is a constant multiple of $\kappa p$ with $p=d \kappa m$. We have $\deg p = \deg m$ and $\ldeg p \geq \ldeg m -1$. Then
$\deg p - \ldeg p \leq \deg m - \ldeg m + 1 \leq r$. 
If $\deg p =r, \deg \kappa m =r+1$ and $\kappa m \in \P_{r+1}\Lambda^k(\R^n)$. On the other hand, if 
$\deg p \geq r+1$, by Proposition~\ref{jspace}, $\kappa m \in \J_{r+1} \Lambda^k(\R^n)$.
This establishes the second inclusion.
To complete the proof, we show that $d \J_{r+1} \Lambda^{k-1}(\R^n) \subset \S_{r+1}\Lambda^k(\R^n)$.
Since $\J_{r+1} \Lambda^{k-1}(\R^n) \subset \S_{r+2} \Lambda^{k-1}(\R^n)$ (by the inclusion just established), we infer
from the subcomplex property that
$d \J_{r+1} \Lambda^{k-1}(\R^n) \subset d \S_{r+2} \Lambda^{k-1}(\R^n) \subset \S_{r+1}\Lambda^k(\R^n)$.
\end{proof}

The third property of the $\S_r\Lambda^k$ spaces that we establish concerns traces on hyperplanes.
Consider a hyperplane $f$ of $\R^n$ of the form $x_i=c$ for some $1\le i \le n$
and some constant $c$.  The variables $x_j$, $j\ne i$, form
a coordinate system for $f$, so we may identify $f$ with $\R^{n-1}$
and consider the space
$\S_r\Lambda^k(f)$.  It is a space of
polynomial $k$-forms on $f$, and so vanishes if $k=n$.
Next, we consider the trace on $f$ of a differential form in $n$ variables
(defined as the pullback of the form through the inclusion map
$f\hookrightarrow\R^n$).
Let $\sigma\in\Sigma(k)$,
and let $\omega_\sigma$ be a function of $n$ variables.  Then
$$
\tr_f(\omega_\sigma\,dx_\sigma) =
\begin{cases}
 0, & i\in \sigma,\\
(\tr_f\omega_\sigma)\, dx_\sigma, & i\notin\sigma.
\end{cases}
$$
In the last expression, $\tr_f\omega_\sigma$
denotes the function of $n-1$ variables obtained by setting $x_i=c$
and we view $dx_\sigma$ as a basic alternator in the $n-1$ variables
$x_j$, $j\ne i$.
The trace property states that if $u\in\S_r\Lambda^k(\R^n)$,
then $\tr_f u$, which is a polynomial $k$-form on $f$,
belongs to $\S_r\Lambda^k(f)$.
\begin{thm}[Trace property]\label{trace}
Let $n,r\ge1$, $0\le k\le n$, and let $f$ be a hyperplane of $\R^n$ obtained by
fixing one coordinate.  Then
$$
\tr_f \S_r\Lambda^k(\R^n) \subset \S_r\Lambda^k(f).
$$
(This inclusion will be shown to be an equality in \eqref{traceeq} below.)
\end{thm}
\begin{proof}
Without loss of generality, we assume that $f=\{\,x\in \R^n\,|\,x_1=c\,\}$.
First let us comment on the Koszul operator applied to a polynomial differential form on $f$.
Such a form may be written as a linear combination of
monomials $x^\alpha\,dx_\sigma$ where $\alpha_1=0$ and $1\notin \sigma$.
Referring to \eqref{defk} we see that, if we view $x^\alpha\,dx_\sigma$ as a form monomial in $n$ variables and take
the Koszul differential, the result is the same as if we view it as form monomial
in $n-1$ variables and take the Koszul differential.  Thus we need not distinguish
between the Koszul differential on $\R^n$ and that on $f$.

We will prove the theorem by induction on $k$.  For $k=0$ we recall that $\S_r\Lambda^0(\R^n)$ is
the serendipity space spanned by the monomials of superlinear degree at most $r$, and,
of course, the superlinear degree does not increase when taking the trace. Hence, $\tr_f \S_r\Lambda^0(\R^n) \subset \S_r\Lambda^0(f)$.

To prove the theorem for $k>0$, assume that it holds with $k$ replaced by $k-1$.
In light of \eqref{defS} we need to show that the traces of each of the three spaces
$\P_r\Lambda^k(\R^n)$, $\J_r\Lambda^k(\R^n)$,
and $d\J_{r+1}\Lambda^{k-1}(\R^n)$ are contained in
$$
\S_r\Lambda^k(f)=  \P_r\Lambda^k(f) + \J_r\Lambda^k(f) +d\J_{r+1}\Lambda^{k-1}(f).
$$
For the $\P_r\Lambda^k(\R^n)$, this is evident, since $\tr_f \P_r\Lambda^k(\R^n)\subset \P_r\Lambda^k(f)$.

Next, we establish that $\tr_f d\J_{r+1}\Lambda^{k-1}(\R^n)\subset \S_r\Lambda^k(f)$.  Indeed,
\begin{multline*}
\tr_f d\J_{r+1}\Lambda^{k-1}(\R^n) = d \tr_f \J_{r+1}\Lambda^{k-1}(\R^n)
\\
\subset d \tr_f \S_{r+1}\Lambda^{k-1}(\R^n) \subset d  \S_{r+1}\Lambda^{k-1}(f)\subset \S_r\Lambda^k(f),
\end{multline*}
where we have used, in turn, the
commutativity of the trace with exterior differentiation, 
\eqref{defS}, the inductive hypothesis, and Theorem~\ref{complex}.

By Proposition~\ref{jspace}, in order to show that $\tr_f\J_r\Lambda^k(\R^n)\subset \S_r\Lambda^k(f)$, and to
complete the proof, it suffices to show that $\tr_f\kappa m\in \S_r\Lambda^k(f)$ whenever $m$ is a $(k+1)$-form
monomial with $\deg m-\ldeg m \le r-1$.  We write $m$ as $x^\alpha\,dx_\sigma$, and consider
separately the cases $1\notin\sigma$ and $1\in\sigma$.

Assuming $1\notin\sigma$, let $p$ be the $(k+1)$-form monomial obtained restricting $m$ to $f$,
i.e., by setting
$x_1=c$ in the coefficient $x^\alpha$.  Then $\tr_f\kappa m = \kappa p$.  If $m$ is linear in $x_1$,
then $\deg p = \deg m-1$ and $\ldeg p =\ldeg m-1$.  Otherwise $\deg p \le \deg m$ and $\ldeg p=\ldeg m$.
In either event, $\deg p-\ldeg p \le \deg m-\ldeg m \le r-1$, so $\kappa p\in \S_r\Lambda^k(f)$ (again using
Proposition~\ref{jspace}).

Assuming, instead, that $1\in\sigma$, we may write
$$
m=x^\alpha\,dx_\sigma = x_1^{\alpha_1}x^\beta\,dx_1\wedge dx_\tau,
$$
where $\beta$ is a multi-index with $\beta_1=0$ and $\tau\in\Sigma(k)$ has $\tau_1>1$.
Then
$$
p:=\tr_f\kappa m = c^{\alpha_1+1} x^\beta\, dx_\tau
$$
is a $k$-form monomial independent of $x_1$ with $\deg p \le \deg m$ and $\ldeg p=\ldeg m$.
We are trying to show that $p\in \S_r\Lambda^k(f)$.  This
is obvious if $\deg p \le r$, so we may assume that $\deg p>r$.  We shall show
that both $\kappa d p$ and $d\kappa p$ belong to $\S_r\Lambda^k(f)$, which suffices by
\eqref{hmt}.

Now $\deg dp = \deg p -1\le \deg m -1$ and $\ldeg dp \ge \ldeg p -1 =\ldeg m -1$,
so $\deg dp - \ldeg dp \le \deg m-\ldeg m\le r-1$.  Therefore $\kappa dp\in \J_r\Lambda^{k}(f)\subset \S_r\Lambda^{k}(f)$,
as required.

Finally, we show that $\kappa p\in \J_{r+1}\Lambda^{k-1}(f)$, whence $d\kappa p \in\S_r\Lambda^k(f)$ as well.
By Proposition~\ref{jspace} this holds, since $\deg p\ge r+1$ and $\deg p-\ldeg p\le \deg m-\ldeg m \le r$
(even $\le r-1$).
This concludes the proof.
\end{proof}

Having defined the space $\S_r\Lambda^k(\R^n)$ of polynomial differential forms, we turn now to the definition
of the associated finite element space on a cubical mesh.  As usual the finite element space is
defined element by element, by specifying a space of shape functions and a set of degrees of freedom on each
cube $T$ in the mesh.  (More generally the element $T$ may be a right rectangular prism, that is, the Cartesian product
of $n$ closed intervals of positive finite length.)
As shape functions on $T$ we use $\S_r\Lambda^k(T)$, the restriction of the above polynomial space
to the cube.  The degrees of freedom have a very simple expression.
Writing $\Delta_d(T)$ for the set of $d$-dimensional faces of $T$, they are given by
\begin{equation}\label{dofs}
\mu\mapsto \int_f \operatorname{tr}_f \mu \wedge \nu, \quad \nu\in \P_{r-2(d-k)}\Lambda^{d-k}(f), \
f\in \Delta_d(T),\ k\le d\le \min(n,\lfloor r/2\rfloor+k).
\end{equation}
Since
$$
\dim \P_{r-2(d-k)}\Lambda^{d-k}(f) = \dim \P_{r-2(d-k)}(f)\x\binom{d}{d-k} = \binom{r-d+2k}{d}\binom{d}{d-k}
$$
for $f$ a face of dimension $d$, and since there are $2^{n-d}\binom{n}{d}$ $d$-dimensional faces of an
$n$-cube, the number of degrees of freedom in \eqref{dofs} is given by
\begin{equation}\label{dim}
\sum_{d=k}^{\min(n,\lfloor r/2\rfloor+k)}  2^{n-d}\binom{n}{d} \binom{r-d+2k}{d}\binom{d}{k}.
\end{equation}

We now turn to one of the main results of this paper, the proof  that the degrees of freedom \eqref{dofs} are unisolvent
for $\S_r\Lambda^k(T)$.
\begin{thm}[Unisolvence]\label{unisolv}
Let $n, r\ge 1$ and $0\le k\le n$, and let $T$ be a cube in $\R^n$.  Then 
\begin{enumerate}
 \item  $\dim\S_r\Lambda^k(T)$ is given by \eqref{dim}.
 \item If $\mu\in \S_r\Lambda^k(T)$ and all the degrees of freedom in \eqref{dofs} vanish, then
$\mu\equiv0$.
\end{enumerate}
\end{thm}

Using the trace property, we will reduce the proof of (2) to the case where $\mu$ belongs
to the space
$$
\0\S_r\Lambda^k(T) := \{ \, 
\mu\in \S_r\Lambda^k(T)\,|\, \tr_f \mu =0\text{ on each face $f \in \Delta_{n-1}(T)$}
\, \},
$$
the subspace with vanishing traces.  This case is given in the following proposition.
\begin{prop} \label{0uni}
If $\mu\in \0\S_r\Lambda^k(T)$
and
\begin{equation}\label{idof}
\int_T \mu\wedge \nu =0,\quad \nu\in \P_{r-2(n-k)}\Lambda^{n-k}(T),
\end{equation}
then $\mu$ vanishes.
\end{prop}
We defer the proof of this proposition to Section~\ref{diffproof}. Now, assuming this result, we prove 
Theorem~\ref{unisolv}.
\begin{proof}
We begin with the first statement of the theorem. 
Since $\J_{r} \Lambda^k(T)$ is in the range of $\kappa$, and since the homotopy formula
implies that no nonzero differential form is in the range of both $\kappa$ and $d$,
the sum on the left of \eqref{defS} is direct.  The homotopy formula implies as well
that $d$ is injective on the range of $\kappa$.  Therefore
$$
\dim \S_r \Lambda^k(T) = \dim \P_r \Lambda^k(T)+
\sum_{l\ge 1} \dim \kappa \H_{r+l-1,l} \Lambda^{k+1}(T)
+\sum_{l\ge 1} \dim \kappa \H_{r+l,l} \Lambda^{k}(T).
$$
Applying Theorem~\ref{t1}, this becomes
\begin{multline*}
\dim \S_r \Lambda^k(T) = \dim \P_r \Lambda^k(T)+\sum_{l\ge 1} \dim \H_{r+l,l} \Lambda^{k}(T)
\\
=\dim[\P_r \Lambda^k(T)+\sum_{l\ge 1} \H_{r+l,l} \Lambda^{k}(T)].
\end{multline*}
The space in brackets is exactly the span of the $k$-form monomials $m$ with $\deg m-\ldeg m \le r$,
and hence we need only count these monomials.  This gives
\begin{equation}\label{temp}
\dim \S_r \Lambda^k(T) =  \binom{n}{k} \# A(r,k,n),
\end{equation}
where $\binom{n}{k}$ is the number of basic alternators and
$A(r,n,k)$ is the set
of monomials $p$ in $n$ variables which are linear in some number $l$ of the
\emph{first} $n-k$ variables, $x_1,\ldots,x_{n-k}$, with $\deg p-l\le r$.
Now we count the elements of $A(r,n,k)$.
For any monomial $p$ in $n$ variables
let $J\subset \N_{n-k}$ be the set of indices for which $x_i$ enters $p$ superlinearly,
let $d\ge0$ be the
cardinality of $J$, and for $i\in \N_{n-k}\setminus J$,
let $a_i=0$ or $1$ according to whether $p$ is of degree $0$ or $1$ in $x_i$.  Then 
$$
p= \bigl(\prod_{j\in J} x_j^2\bigr)\x q \x \bigl(\prod_{i\in\N_{n-k}\setminus J } x_i^{a_i}\bigr),
$$
where $q$ is a monomial in the $d$ variables indexed by $J$ and the last $k$ variables.  With
$l$ the number of $a_i$ equal to $1$, we have $\deg p = 2d+\deg q +l$.
Thus $\deg p-l\le r$ if and only if $\deg q\le r-2d$.
Thus we may uniquely specify an element of $A(r,k,n)$
by choosing $d\ge 0$, choosing the set $J$ consisting of $d$ of the $n-k$ variables
(for which there are $\binom{n-k}{d}$ possibilities), choosing the monomial $q$
of degree at most $r-2d$ in the $d+k$ variables ($\binom{r-d+k}{d+k}$
possibilities), and choosing the exponent $a_i$ to be
either $0$ or $1$ for the $n-k-d$ remaining indices
($2^{n-k-d}$ possibilities).  Thus
\begin{align}\label{pdim2}
\begin{split}
\#A(r,k,n) & = \sum_{d=0}^{\min(n-k,\lfloor r/2\rfloor)} 2^{n-k-d}\binom{n-k}{d}
\binom{r-d+k}{d+k} \\
&= \sum_{d=k}^{\min(n,\lfloor r/2\rfloor+k)} 2^{n-d}\binom{n-k}{d-k}
\binom{r-d+2k}{d},
\end{split}
\end{align}
where the second sum comes from a change of the summation index ($d\to d-k$).
Substituting \eqref{pdim2} into \eqref{temp} and using the binomial identity
\begin{align*}
\binom{n}{k} \binom{n-k}{d-k} =  \binom{d}{k} \binom{n}{d},
\end{align*}
we conclude that
\begin{equation}
\dim \S_r\Lambda^k(T)  = \sum_{d=k}^{\min(n,\lfloor r/2\rfloor+k)}  2^{n-d}\binom{n}{d} \binom{r-d+2k}{d}\binom{d}{k}.
\end{equation}
This completes the proof of the dimension formula for $\S_r\Lambda^k(\R^n)$.

The proof of unisolvence is easily completed based on the trace property and
Proposition~\ref{0uni}.  We use induction on the dimension $n$, the one-dimensional case being
trivial.  Suppose $\mu\in\S_r\Lambda^k(T)$ and all its degrees of freedom vanish.  For any face
$f$ of dimension $n-1$,  $\tr_f \mu\in S_r\Lambda^k(f)$ and all the degrees
of freedom for it vanish.  By induction $\tr_f \mu\equiv 0$ on $f$.  This
implies that $\mu\in \0\S_r\Lambda^k(T)$, and we invoke Proposition~\ref{0uni} to conclude that $\mu$
vanishes identically.
\end{proof}

We remark that, as a corollary of unisolvence, we may strengthen the result of
Theorem~\ref{trace} to equality
\begin{equation}\label{traceeq}
\tr_f \S_r\Lambda^k(\R^n) = \S_r\Lambda^k(f).
\end{equation}
Indeed, let $T$ be a cube with one face contained in the hyperplane $f$.  Then $T\cap f$ is
an $(n-1)$-dimensional cube and any $\nu\in\S_r\Lambda^k(f)$ is uniquely
determined by the degrees of freedom for the $\S_r\Lambda^k(T\cap f)$.  Now we may determine
an element $\mu\in \S_r\Lambda^k(\R^n)$ by assigning the degrees of freedom for the space $\S_r\Lambda^k(T)$
arbitrarily.  In particular, we may choose the values of those degrees of freedom associated to the face $T\cap f$
and its subfaces to be the same as those for $\nu$.  Then
$\tr_f\mu \in \S_r\Lambda^k(f)$ (by Theorem~\ref{trace}), and $\tr_f\mu$ and $\nu$ have identical
degrees of freedom, and so they are equal (by Theorem~\ref{unisolv}).

With the definition of the spaces complete, we use the subcomplex property, Theorem~\ref{complex},
to define a subcomplex of the de~Rham complex on the cube
$$
\begin{CD}
\R @>\subset>> \S_r\Lambda^0(T) @>d>> \S_{r-1}\Lambda^1(T) @>d>> \cdots @>d>> \S_{r-n}\Lambda^n(T) @>>>0.
\end{CD}
$$
To show that this complex is exact, we define the canonical projection
$$
\pi_r^k: C\Lambda^k(T)\to \S_r\Lambda^k(T),
$$
associated to the unisolvent degrees of freedom.  That is, $\pi_r^k\mu\in\S_r\Lambda^k(T)$ is
determined by the equations
\begin{multline*}
\int_f\tr_f(\pi_r^k\mu) \wedge \nu = \int_f\tr_f\mu\wedge\nu,
\quad \nu\in \P_{r-2(d-k)}\Lambda^{d-k}(f), \
f\in \Delta_d(T),\\ k\le d\le \min(n,\lfloor r/2\rfloor+k).
\end{multline*}
Then, the following diagram commutes:
$$
\begin{CD}
 \R @>\subset>> C^\infty\Lambda^0(T) @>d>> C^\infty\Lambda^1(T) @>d>> \cdots @>d>> C^\infty\Lambda^n(T) @>>> 0
\\
@. @VV\pi_r^0V @VV\pi_{r-1}^1V @. @VV\pi_{r-n}^nV
\\
\R @>\subset>> \S_r\Lambda^0(T) @>d>> \S_{r-1}\Lambda^1(T) @>d>> \cdots @>d>> \S_{r-n}\Lambda^n(T) @>>>0
\end{CD}
$$
The proof of commutativity is based on two basic properties of differential forms:
(1)~the commutativity of trace and exterior differentiation, $\tr_f d\omega=d_f\tr_f\omega$,
and (2)~integration by parts, which for differential forms
can be written as
$$
\int_\Omega d\omega\wedge \eta = (-1)^{k-1}\int_\Omega\omega\wedge d\eta + \int_{\partial\Omega}
\tr_{\partial\Omega}\omega\wedge\tr_{\partial\Omega}\eta,
$$
for a $k$-form $\omega$ and an $(n-k-1)$-form $\eta$ on an $n$-dimensional domain $\Omega$.
See \cite[Lemma~4.24]{acta} for the same argument applied to simplicial elements.
Since the top row of the diagram, the de~Rham complex on the cube, is exact, the commutativity of
the diagram implies that the bottom row is exact as well.

Having defined the finite element space $\S_r\Lambda^k(T)$ on a single cube $T$ and established
its properties, the space $\S_r\Lambda^k(\T_h)$ associated to a cubical mesh $\T_h$ is defined
through the usual finite element assembly.  In view of the unisolvence result Theorem~\ref{unisolv}
and the trace result \eqref{traceeq}, the degrees of freedom
associated to a face of the cube and its subfaces determine the trace of the finite
element differential form on the face.  It follows that
$\S_r\Lambda^k(\T_h)\subset H \Lambda^k(\Omega)$ (see \cite[Section~5.1]{acta}).

\section{Unisolvence over the space with vanishing traces} \label{diffproof}
We conclude the paper with the proof of Proposition~\ref{0uni},
which is based on the following lemma.
\begin{lemma}\label{lem1}
 Suppose that $\eta=\sum_{\sigma\in\Sigma(k)}\eta_\sigma\,dx_\sigma$
where, for
each $\sigma\in\Sigma(k)$, $\eta_\sigma$ is a homogeneous polynomial
which is superlinear in all the $\sigma^*$ variables.
Further
suppose that $\ldeg \kappa\eta\ge 1$ and $\ldeg \kappa d\eta\ge 1$.  Then $\eta=0$.
\end{lemma}
\begin{proof}[Proof of Lemma~\ref{lem1}]
  By \eqref{kdformula} we have
$\kappa d\eta = \sum_{\sigma\in\Sigma(k)}\mu_\sigma\,dx_\sigma$,
where
\begin{equation}\label{mus}
\mu_\sigma = \sum_{q\in\sigma^*}\bigl[x_q\partial_q\eta_\sigma +
\sum_{p\in\sigma}\epsilon(q,p,\sigma)x_q\partial_p\eta_{\sigma+q-p}\bigr].
\end{equation}
Since $\ldeg \kappa d\eta\ge1$, each monomial of the polynomial $\mu_\sigma$
is linear in at least one $\sigma^*$ variable.

For any $\sigma\in\Sigma(k)$ and any subset $\tau$ of $\sigma$,
let $S_\sigma^\tau$ be the span of the (ordinary, $0$-form) monomials which are independent
of the $\tau$ variables but depend on all of the other $\sigma$ variables,
and let $S_\sigma^*$ denote the span of the monomials which are superlinear
in all the $\sigma^*$ variables.  Denote by $P_\sigma^\tau$ the projection
onto $S_\sigma^\tau$.  That is, if $p=\sum p_m m$ where the sum is
over all monomials $m$ and the coefficients $p_m$ are real numbers, all but
finitely many zero, then $P_\sigma^\tau p:=\sum_{m\in S_\sigma^\tau}p_m m$.
Similarly, we denote by $Q_\sigma^\tau$ the projection
onto $S_\sigma^\tau\cap S_\sigma^*$.  We now calculate the result of applying $Q_\sigma^\tau$
to both sides of \eqref{mus}.  First we note that
\begin{equation}\label{q0}
 Q_\sigma^\tau\mu_\sigma=0,
\end{equation}
since every monomial of $\mu_\sigma$ is linear in at least one $\sigma^*$ variable,
and so none of them belong to $S_\sigma^*$.
Next, for each $q\in\sigma^*$,
\begin{equation}\label{q1}
Q_\sigma^\tau(x_q\partial_q\eta_\sigma)= P_\sigma^\tau(x_q\partial_q\eta_\sigma)
= x_q\partial_q (P_\sigma^\tau \eta_\sigma),
\end{equation}
with the first inequality holding since each monomial of $\eta_\sigma$, and therefore also
of $x_q\partial_q\eta_\sigma$ is superlinear in all the $\sigma$ variables.
Now we determine the action of $Q_\sigma^\tau$ on
the terms of the second sum on the right-hand side of \eqref{mus}.
For any $q\in\sigma^*$ and $p\in\sigma$, we claim that
\begin{equation}\label{q2}
Q_\sigma^\tau(x_q\partial_p\eta_{\sigma+q-p})
= 
\begin{cases}
0, & p\in\tau,\\
  \displaystyle x_q\partial_p (P_{\sigma+q-p}^\tau\eta_{\sigma+q-p}), & p\in\sigma\setminus\tau.
\end{cases}
\end{equation}
Indeed, $\eta_{\sigma+q-p}$ is superlinear in $x_p$, so every monomial
of $x_q\partial_p\eta_{\sigma+q-p}$ depends on $x_p$.  This implies
that the projection is $0$ in the case $p\in\tau$.  In case $p\in\sigma\setminus\tau$,
we write
$x_q\partial_p\eta_{\sigma+q-p} = \sum x_q\partial_p m$,
where the sum is over the monomials $m$ of $\eta_{\sigma+q-p}$.
Since neither $p$ or $q$ belongs to $\tau$, the monomial $m':=x_q\partial_p m$
is independent of the $\tau$ variables if and only if the same is true of the monomial $m$,
and, since $m'$ always depends on $x_p$,
it depends on all of the  $\sigma\setminus\tau$ variables
if and only if $m$ depends on all of the $(\sigma-p)\setminus\tau$ variables.  Further,
$m'$ is always superlinear in all the $\sigma^*$
variables except possibly $x_q$, and it is superlinear in $x_q$ if
and only if $m$ depends on $x_q$.  In short, $m'$ belongs to $S_\sigma^\tau\cap S_\sigma^*$
if and only if $m$ belongs to $S_{\sigma+q-p}^\tau$.  This completes the verification of \eqref{q2}.
Thus the application of $Q_\sigma^\tau$ to \eqref{mus} gives, in light of \eqref{q0}, \eqref{q1},
and \eqref{q2}, that
\begin{equation}\label{qmus}
\sum_{q\in\sigma^*}\bigl[x_q\partial_q(P_\sigma^\tau\eta_\sigma) +
\sum_{p\in\sigma\setminus\tau}\epsilon(q,p,\sigma)
 x_q\partial_p ( P_{\sigma+q-p}^\tau\eta_{\sigma+q-p})\bigr]
=0.
\end{equation}

We now claim that, for any $\sigma\in\Sigma(k)$ and $\tau\subset\sigma$, that
\begin{equation}\label{sumi}
\sum_{q\in\sigma^*}\sum_{p\in\sigma\setminus\tau}\epsilon(q,p,\sigma)
 x_q\partial_p (P_{\sigma+q-p}^\tau\eta_{\sigma+q-p})
=\bigl(c +\sum_{i\in\sigma\setminus\tau} x_i \partial_i\bigr)
(P_\sigma^\tau\eta_\sigma),
\end{equation}
where $c=\#(\sigma\setminus\tau)$.
Assuming this, we have from \eqref{qmus}, 
$$
0=\bigl(c +\sum_{i\in\sigma\setminus\tau} x_i \partial_i
+\sum_{q\in\sigma^*}x_q\partial_q\bigr)(P_\sigma^\tau\eta_\sigma)
= \bigl(c +\sum_{i\in\tau^*} x_i \partial_i\bigr)(P_\sigma^\tau\eta_\sigma)
=(c+r+1)(P_\sigma^\tau\eta_\sigma),
$$
where we have used Euler's formula for homogeneous polynomials
(i.e., the homotopy formula for $0$-forms)
in the last step.  Thus, $P_\sigma^\tau\eta_\sigma$ vanishes, and so
$\eta_\sigma=\sum_{\tau\subset\sigma}P_\sigma^\tau\eta_\sigma$ also vanishes.
Since $\sigma\in\Sigma(k)$ is arbitrary, this implies that $\eta=0$.

Thus it remains only to prove \eqref{sumi}. 
By the first formula of Lemma~\ref{kform},
$\kappa\eta = \sum_{\zeta\in\Sigma(k-1)} \omega_\zeta \,dx_\zeta$ where
$$
\omega_\zeta = \sum_{q\in\zeta^*} \epsilon(q,\zeta)x_q\eta_{\zeta+q}.
$$
Now let $\zeta\in\Sigma(k-1)$ and
$\tau\subset\zeta$, and denote
by $R_\zeta^\tau$ the projection onto $S_{\zeta}^\tau\cap S_\zeta^*$ (the span
of monomials which depend on all the $\zeta$ variables except the $\tau$ variables
and which are superlinear in the $\zeta^*$ variables).  By the hypothesis
that $\ldeg\kappa\eta\ge 1$,
\begin{equation*}
 R_\zeta^\tau\omega_\zeta=0.
\end{equation*}
Next we compute $R_\zeta^\tau (x_q\eta_{\zeta+q})$ for $q\in\zeta^*$.
If $m$ is a monomial of $\eta_{\zeta+q}$, then $m'= x_q\eta_{\zeta+q}$
depends on all the $\zeta$ variables except
for the $\tau$ variables if and only if the same is true of $m$.
Moreover, $m'$ is superlinear in all the $\zeta^*$ variables
if and only if $m$ depends on $x_q$ (since $m$ is superlinear in all
the $\zeta^*$ variables with the possible exception of $x_q$).  Thus,
\begin{equation*}
R_\zeta^\tau (x_q\eta_{\zeta+q}) = x_q P_{\zeta+q}^\tau \eta_{\zeta+q}.
\end{equation*}
Combining the last three displayed equations, we obtain
$$
\sum_{q\in\zeta^*}\epsilon(q,\zeta) x_q P_{\zeta+q}^\tau \eta_{\zeta+q}=0.
$$

Now choose some $p\in\zeta^*$ and differentiate this equation with respect
to $x_p$ to get
$$
\epsilon(p,\zeta) \partial_p(x_p P_{\zeta+p}^\tau\eta_{\zeta+p} )
+ \sum_{q\in\zeta^*-p} \epsilon(q,\zeta) x_q \partial_p
(P_{\zeta+q}^\tau \eta_{\zeta+q})=0,
$$
or, after rearranging,
\begin{equation}\label{t}
P_{\zeta+p}^\tau\eta_{\zeta+p} + x_p \partial_p(P_{\zeta+p}^\tau\eta_{\zeta+p} )
=-\sum_{q\in\zeta^*-p} \epsilon(p,\zeta)\epsilon(q,\zeta) x_q \partial_p
(P_{\zeta+q}^\tau \eta_{\zeta+q}).
\end{equation}

For any $\sigma\in\Sigma(k)$ and any $p\in\sigma$ set $\zeta=\sigma-p\in\Sigma(k-1)$.
Then we can rewrite \eqref{t} in terms of $\sigma$ as
$$
P_\sigma^\tau\eta_\sigma + x_p \partial_p(P_\sigma^\tau\eta_\sigma )
=-\sum_{q\in\sigma^*} \epsilon(p,\sigma-p)\epsilon(q,\sigma-p) x_q \partial_p
(P_{\sigma+q-p}^\tau \eta_{\sigma+q-p}).
$$
Finally, taking any $\tau\subset\sigma\in\Sigma(k)$, we sum over $p\in\sigma\setminus\tau$
to obtain
$$
 c P_\sigma^\tau\eta_\sigma
 + \sum_{p\in\sigma\setminus\tau} x_p \partial_p(P_\sigma^\tau\eta_\sigma )
=-\sum_{p\in\sigma\setminus\tau} \sum_{q\in\sigma^*}
 \epsilon(p,\sigma-p)\epsilon(q,\sigma-p) x_q \partial_p
(P_{\sigma+q-p}^\tau \eta_{\sigma+q-p}),
$$
where $c=\#(\sigma\setminus\tau)$.  In light of \eqref{epseps}, this establishes
\eqref{sumi}, and so completes the proof of the lemma.
\end{proof}

Finally, we give the proof of Proposition~\ref{0uni}.
\begin{proof}
By dilating and translating, it suffices to prove the result when $T=I^n$
with $I=[-1,1]$.
Let
$$
\mu= \sum_{\sigma \in \Sigma(k)} \mu_{\sigma} d x_{\sigma}
$$
be a polynomial differential form on the cube $I^n$.  Then $\tr_f \mu$ vanishes on
the faces $x_i=\pm 1$ if and only if for each $\sigma$ such that $i\notin\sigma$,
$1-x_i^2$ divides $\mu_{\sigma}$.  Now suppose that $\mu\in \0\S_r\Lambda^k(I^n)$,
so that $\tr_f \mu$ vanishes on all the faces of the cube.  It follows that
$$
\mu_{\sigma} = \tilde{\mu}_{\sigma} \prod_{i \in \sigma^*} (1-x_i^2)
$$
for some polynomial
$\tilde{\mu}_{\sigma}$.
The monomial expansion of $\mu$ then contains the form monomial
$m_\sigma\prod_{i \in \sigma^*} x_i^2\,dx_\sigma$, where $m_\sigma$ is any monomial of highest degree of $\tilde \mu_\sigma$.
The linear degree of this form monomial is $0$, so, by the degree property \eqref{deg},
its degree is at most $r+1$. 

Having established that $\mu$ is of degree at most $r+1$, let
$\eta$ be its homogeneous part of degree $r+1$. We have $\ldeg \eta = 0$.
Now we may match terms in the definition \eqref{defS} of $\S_r\Lambda^k(\Omega)$
to obtain that
$\eta = \kappa\upsilon + d \kappa\omega$, for some
$\upsilon \in \H_{r,1}\Lambda^{k+1}(I^n)$ and $\omega\in \H_{r+1,1}\Lambda^{k}(I^n)$.
Therefore, $\ldeg\kappa\eta=\ldeg\kappa d\kappa\omega\ge\ldeg\omega\ge 1$ and $\ldeg\kappa d\eta
=\ldeg\kappa d\kappa\upsilon\ge\ldeg\upsilon\ge 1$, where we have used Lemma~\ref{degs}.

By Lemma ~\ref{lem1}, $\eta=0$, hence the monomial of highest order in the expansion of $\mu_{\sigma}$ is of degree at most $r$.
It follows that $\tilde{\mu}_{\sigma}$ is of degree $r-2(n-k)$ for each $\sigma \in \Sigma(k)$. We can then choose the test function
$\nu =  \sum_{\sigma \in \Sigma(k)} (-1)^{\text{sgn}(\sigma,\sigma^*)}\tilde{\mu}_{\sigma} d x_{\sigma^*}$ in \eqref{idof} to conclude that $\mu$ vanishes.
\end{proof}

\bibliographystyle{amsplain}
\bibliography{cubicderham}
\end{document}